\documentclass{amsart}
\usepackage{amsmath,amssymb,latexsym,hyperref,url,graphicx}

\newtheorem{theorem}{Theorem}
\newtheorem{corollary}[theorem]{Corollary}
\newtheorem{lemma}[theorem]{Lemma}

\begin{document}

\title{The rational--transcendental dichotomy of Mahler functions}
\author{Jason P.~Bell}
\address{
	Department of Mathematics \\
	Simon Fraser University \\
	Burnaby, BC \\
	Canada
}
\author{Michael Coons}
\address{
	School of Mathematical and Physical Sciences \\
	The University of Newcastle \\
	Callaghan, NSW \\
	Australia
}
\author{Eric Rowland}
\address{
	LaCIM \\
	Universit\'e du Qu\'ebec \`a Montr\'eal \\
	Montr\'eal, QC \\
	Canada
}

\begin{abstract}
In this paper, we give a new proof of a result due to B\'ezivin that a $D$-finite Mahler function is necessarily rational. This also gives a new proof of the rational--transcendental dichotomy of Mahler functions due to Nishioka. Using our method of proof, we also provide a new proof of a P\'olya--Carlson type result for Mahler functions due to Rand\'e; that is, a Mahler function which is meromorphic in the unit disk is either rational or has the unit circle as a natural boundary.
\end{abstract}

\maketitle

\section{Introduction}

The {\em Thue--Morse sequence} $\{t(n)\}_{n\geq 0}$ over the alphabet $\{-1,1\}$ is given by $t(n)=(-1)^{s(n)}$ where $s(n)$ is the number of $1$s in the base $2$ expansion of the number $n$. Using this definition it is immediate that the sequence $\{t(n)\}_{n\geq 0}$ is $2$-automatic. That is, there is a deterministic finite automaton that takes the base $2$ expansion of $n$ as input and outputs the value $t(n)$ (see Figure \ref{Fig1}); the definition of an automatic sequence will be discussed in more detail below.

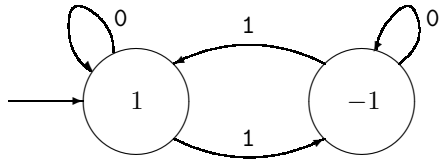
\begin{figure}[htbp]
\begin{center}
\setlength{\unitlength}{1mm}
\begin{picture}(60,20)
  \put(14.25,4){{$1$}}
  \put(15,5){\circle{20}}
  \put(43,4){{$-1$}}
  \put(45,5){\circle{20}}
  
  \put(-2,5){\vector(1,0){10}} 

  \qbezier(40, 10)(30, 15)(20, 10)
  \put(21, 10.5){\vector(-2,-1){1}}
    \put(29,14){{\tt 1}}
    
  \qbezier(40, 0)(30, -5)(20, 0)
  \put(39, -.5){\vector(2,1){1}}
    \put(29,-1){{\tt 1}}
      
  \qbezier(12, 11.5)(12, 15)(9, 17)
  \qbezier(9, 17)(6, 19)(6, 15)
  \qbezier(6, 15)(6, 11)(9, 9)
  \put(8.25, 10){\vector(1,-1){1}}
    \put(12,15){{\tt 0}}
  
  \qbezier(50, 10)(54, 13)(52, 17)
  \qbezier(52, 17)(51, 19)(48, 15)
  \qbezier(48, 15)(47, 13.3)(47, 12)
  \put(47.6, 13.6){\vector(-1,-2){1}}
    \put(53.5,15){{\tt 0}}
\end{picture}
\caption{The generating $2$-automaton of the Thue--Morse sequence.}
\label{Fig1}
\end{center}
\end{figure}

\noindent It is also immediate using the definition that the Thue--Morse sequence is the unique sequence given by $t(0)=1$, $t(2n)=t(n)$ and $t(2n+1)=-t(n)$. Thus, writing the generating power series for $t(n)$ as $T(z)=\sum_{n\geq 0}t(n)z^n$, we have that $T(z)$ satisfies the functional equation $$(1-z)T(z^2)=T(z).$$ The {\em Thue--Morse number} is the evaluation of the generating power series for $\{t(n)\}_{n\geq 0}$ at $z=1/2$; that is, the number $T(1/2)=\sum_{n\geq 0}t(n)2^{-n}.$ Due to the nature of the construction of the Thue--Morse sequence and the functional equation satisfied by its generating power series, the Thue--Morse number and those numbers with similar attributes are a natural class of numbers to consider from the perspectives of transcendental number theory and Diophantine approximation. 

Focusing on the functional equation aspect, in the late 1920s and early 1930s, Mahler \cite{M1929, M1930a, M1930b} showed that if $F(z)\in\mathbb{Q}[[z]]$ satisfies \begin{equation}\label{LF}a(z)F(z^k)=b(z)F(z)\end{equation} for some $a(z),b(z)\in\mathbb{Q}[z]$ and $F(z)$ is transcendental over $\mathbb{Q}(z)$, then for all but finitely many algebraic points $\alpha$ in radius of convergence of $F(z)$, $F(\alpha)$ is transcendental over $\mathbb{Q}$. One of the main goals of Mahler's above-cited work was to provide a proof of the transcendence of the Thue--Morse number, a goal which Mahler accomplished \cite{M1929}. Extensions of Mahler's work in this area have been, and are still, a vital area of research in number theory under the heading {\em Mahler's method}. In this paper, we are concerned with the following generalization of~\eqref{LF}. 

Let $k\geq 2$ be a positive integer. We say that a function $F(z)\in\mathbb{C}[[z]]$ is {\em $k$-Mahler} (or sometimes just {\em Mahler}) provided there exist a non-negative integer $d$ and polynomials $a_0(z),\ldots,a_d(z)\in\mathbb{C}[z]$ with $a_0(z)a_d(z)\neq 0$ such that \begin{equation}\label{F} a_0(z)F(z)+a_1(z)F(z^k)+\cdots+a_d(z)F(z^{k^d})=0.\end{equation} The functional equation in \eqref{F} is called a {\em Mahler-type functional equation}. Loxton and van der Poorten \cite{LP1988} claimed an analogue of Mahler's above-mentioned result for functions $F(z)\in\mathbb{Z}[[z]]$ satisfying \eqref{F}, although their proof is not complete \cite[remark following Corollary~2]{B1994}.

One may also wish to possibly avoid the functional equation approach and focus on generalizing the automatic aspect inspired by the Thue--Morse sequence. While the definition of automaticity based on computability by finite state automata is useful in a computational setting, we will use an equivalent and combinatorial (or sequential) definition which is more suitably generalized in a mathematical context. Let $\mathbf{f}:=\{f(n)\}_{n\geq 0}$ be a sequence with values in a ring $R$. We define the {\em $k$-kernel} of $\mathbf{f}$ to be the set $$\Big\{\{f(k^ln+r)\}_{n\geq 0}:l\geq 0 \mbox{ and } 0\leq r<k^l\Big\}.$$ Given $k\geq 2$, we say a sequence $\mathbf{f}$ is $k$-{\em automatic} if and only if the $k$-kernel of $\mathbf{f}$ is finite. Note that by the definition, an automatic sequence can only take on a finite number of values. To expand the notion of automatic sequences to sequences which can have values from an infinite set, we take the following definition from Allouche and Shallit \cite{AS1992}. We say that a sequence $\mathbf{f}$ taking values in a $\mathbb{Z}$-module $R$ is a {\em $k$-regular sequence} (or just $k$-regular) provided there exist a finite number of sequences over $R$, say $\{\mathbf{f}_1,\ldots,\mathbf{f}_m\}$, with the property that every sequence in the $k$-kernel of $\mathbf{f}$ is a $\mathbb{Z}$-linear combination of the $\mathbf{f}_i$; that is, $\mathbf{f}$ is $k$-regular provided the $k$-kernel of $\mathbf{f}$ is contained in a finitely generated $\mathbb{Z}$-module. We say that the series $\sum_{n\geq 0}f(n)z^n$ is {\em $k$-regular} (resp.~$k$-automatic) when $\mathbf{f}$ is $k$-regular (resp.~$k$-automatic). Note that with these definitions in mind, it is easy to see that any $k$-automatic function is also $k$-regular.

While the generalizations put forth in the preceding two paragraphs seem to branch into unrelated territories, they are in fact closely related. Indeed, in 1994 Becker \cite{B1994} showed that a $k$-regular series is $k$-Mahler; this was previously known for $k$-automatic series \cite{LP1988}. Thus to consider questions of transcendence of regular or automatic numbers, it may be enough to focus on Mahler's method. In the application of Mahler's method, as in the approach of Loxton and van der Poorten~\cite{LP1988}, one must have in hand a transcendence result for the series $F(z)$. Fortunately, Nishioka~\cite[Theorem 5.1.7]{N1996} established a rational--transcendental dichotomy for solutions of certain systems of functional equations; in particular, a $k$-Mahler function $F(z)\in\mathbb{C}[[z]]$ is either rational or transcendental.
Also in 1994, B\'ezivin~\cite[Th\'eor\`eme~1-3]{Bezivin} provided an extension of the special case of Nishioka's theorem to differentiably finite power series.
Recall that a function is said to be {\em differentiably finite} (or {\em $D$-finite}) if it satisfies a homogeneous linear differential equation with polynomial coefficients~\cite{S1980}.

\begin{theorem}[B\'ezivin \cite{Bezivin}]\label{Df} Let $k\geq 2$ be an integer and $F(z)\in\mathbb{C}[[z]]$ be a $k$-Mahler function. If $F(z)$ is $D$-finite, then $F(z)$ is a rational function.
\end{theorem}

Rand\'e in his doctoral thesis \cite{R1992} proved a P\'olya--Carlson type theorem for Mahler functions.
Unfortunately, Rand\'e's result has not appeared in the literature.

\begin{theorem}[Rand\'e \cite{R1992}]\label{PC} Let $k\geq 2$ be an integer and $F(z)\in\mathbb{C}[[z]]$ be a $k$-Mahler function. Then $F(z)$ is a rational function or it has the unit circle as a natural boundary.
\end{theorem}

In this paper, we provide new proofs of the theorems of B\'ezivin and Rand\'e.
These theorems appear here in English for the first time.
We note that the main ingredient of our proofs of Theorems \ref{Df} and \ref{PC} is that a meromorphic $k$-Mahler function is necessarily rational.

\section{Rational--transcendental dichotomy}



We use the following result of Dumas \cite[Theorem 31]{D1993}.

\begin{theorem}[Structure Theorem of Dumas \cite{D1993}]\label{Dumas} A $k$-Mahler function is the quotient of a series and an infinite product which are $k$-regular and analytic in the unit disk. That is, if $F(z)$ is the solution of the Mahler functional equation $$a_0(z)F(z)+a_1(z)F(z^k)+\cdots+a_d(z)F(z^{k^d})=0,$$ where $a_0(z)a_d(z)\neq 0$, the $a_i(z)$ are polynomials, then there exists a $k$-regular series $H(z)$ such that $$F(z)=\frac{H(z)}{\prod_{j\geq 0}\Gamma(z^{k^j})},$$ where $a_0(z)=\rho z^{\delta_0}\Gamma(z)$, with $\rho\neq 0$ and $\Gamma(0)=1$.
\end{theorem}

Dumas' theorem yields the following immediate corollary, which we note here as a lemma. We denote the open ball of radius $r>0$ centered at the origin by $B(0,r)$.

\begin{lemma} Let $k\geq 2$ be an integer and let $F(z)\in\mathbb{C}[[z]]$ be a $k$-Mahler function. Then $F(z)$ has a positive radius of convergence. 
\end{lemma}

\begin{proof} Let $k\geq 2$ be an integer and $F(z)\in\mathbb{C}[[z]]$ be a $k$-Mahler function satisfying, say, $$\sum_{j=0}^{d} a_j(z)F(z^{k^j})=0,$$ for $a_j(z)\in\mathbb{C}[z]$ with $a_0(z)a_d(z)\neq 0$. Noting that a $k$-regular series is analytic in the unit disk (see Allouche and Shallit \cite[Theorem 2.10]{AS1992}), Theorem \ref{Dumas} gives that $F(z)$ converges in $B(0,r)$, where $r\in(0,1)$ is the minimal distance from $0$ to a nonzero root of $a_0(z)(z-1)$.
\end{proof}

\begin{lemma}\label{merofinite} Let $k\geq 2$ be an integer and let $F(z)\in\mathbb{C}[[z]]$ be a $k$-Mahler function. The function $F(z)$ is meromorphic if and only if it has finitely many singularities. Moreover, if $F(z)$ is not meromorphic then it has infinitely many non-polar singularities on the unit circle.
\end{lemma}

\begin{proof} Let $k\geq 2$ be an integer and $F(z)\in\mathbb{C}[[z]]$ be a $k$-Mahler function satisfying, say, $$\sum_{j=0}^{d} a_j(z)F(z^{k^j})=0,$$ for $a_j(z)\in\mathbb{C}[z]$ with $a_0(z)a_d(z)\neq 0$. Write \begin{equation}\label{F1} F(z^{k^d})=-\sum_{j=0}^{d-1} \frac{a_j(z)}{a_d(z)}F(z^{k^j}).\end{equation} 

Suppose that $F(z)$ is not meromorphic on the plane, and let $\alpha=Re^{i\vartheta}$ be a non-polar singularity of $F(z)$ with $R\geq 1$ minimal and $\vartheta\in[0,2\pi);$ note that such a minimal singularity exists since the singularity set is closed and $F(z)$ has only polar singularities in the unit circle (see Theorem \ref{Dumas}).

We note here that the case $R>1$ cannot occur. To see this, suppose that $z=Re^{i\vartheta_0}$ is a non-polar singularity of $F(z)$ of minimal distance $R>1$ to the origin and $\vartheta_0\in[0,2\pi)$. Then $F(z^{k^d})$ has a non-polar singularity at $z=R^{k^{-d}}e^{i\vartheta_0 k^{-d}},$ and so by \eqref{F1} it must be the case that the right-hand side of \eqref{F1} has a non-polar singularity at $z=R^{k^{-d}}e^{i\vartheta_0 k^{-d}}.$ Since this cannot be contributed by the rational functions, there is some $j_0\in\{0,\ldots,d-1\}$ such that $F(z^{k^{j_0}})$ has a non-polar singularity at $z=R^{k^{-d}}e^{i\vartheta_0 k^{-d}},$ which in turns implies that $F(z)$ has a non-polar singularity at $z=R^{k^{-d+j_0}}e^{i\vartheta_0 k^{-d+j_0}}.$ Since $R^{k^{-d+j_0}}<R$ and $R$ was chosen minimally we arrive at a contradiction. Thus it must be the case that $R=1$. 

Supposing that $R=1$, we will show that $F(z)$ has infinitely many singularities on the unit circle. To this end suppose that $F(z)$ has a non-polar singularity at $z=e^{i\vartheta_0}$ with $\vartheta_0\in(0,2\pi]$. 
Then $F(z^{k^d})$ has a non-polar singularity at $z=e^{i\vartheta_0{k^{-d}}},$ and so by \eqref{F1} it must be the case that the right-hand side of \eqref{F1} has a non-polar singularity at $z=e^{i\vartheta_0k^{-d}}.$ Since this cannot be contributed by the rational functions, there is some $j_0\in\{0,\ldots,d-1\}$ such that $F(z^{k^{j_0}})$ has a non-polar singularity at $z=e^{i\vartheta_0k^{-d}},$ which in turns implies that $F(z)$ has a non-polar singularity at $z=e^{i\vartheta_0k^{-d+j_0}}.$ Now set $$\vartheta_1:=\vartheta_0k^{-d+j_0},$$ and repeat this process to construct an infinite sequence of distinct non-polar singularities $\{e^{i\vartheta_n}\}_{n\geq 0}$ of $F(z)$ with the property that \begin{equation}\label{limRn}\lim_{n\to\infty} e^{i\vartheta_n}=1.\end{equation} Thus if $F(z)$ has finitely many singularities, it is meromorphic. Note also, we have shown that if $F(z)$ has a non-polar singularity, then it has a non-polar singularity at $z=1$, since the singularity set is closed and the polar singularities are isolated.

On the other hand, suppose that $F(z)$ is meromorphic and for the sake of a contradiction, suppose that it has infinitely many singularities. We may assume that there is a sequence of these singularities, say $\{R_ne^{i\vartheta_n}\}_{n\geq 0}$, which tend to infinity in modulus, since if not there would be an accumulation point of singularities and that accumulation point would be a non-isolated singularity, which would contradict that $F(z)$ is meromorphic. We will show that under these conditions the function $F(z)$ has infinitely many singularities in a bounded region. To this end, let $L>1$ be fixed, but large enough so that all zeros of the polynomials $a_0(z),\ldots,a_d(z)$ have modulus strictly less than $L$. Let $M\in\mathbb{N}$ be such that $M>1$ and $\alpha=Re^{i\vartheta}$ is a singularity of $F(z)$ with $L^{k^{Md}}\leq R<L^{k^{(M+1)d}}$ and $\vartheta\in(0,2\pi]$. Define the region $$\mathcal{S}(M):=\left\{z\in\mathbb{C}: |z|\in[L,L^{k^{2d}}], \arg z\in\left(0,\frac{2\pi}{k^{(M-1)d}}\right]\right\}.$$ Note that we have assumed without loss of generality that there are infinitely many such $M$ (and $\alpha$) with the above constraints and so to prove this direction of the lemma it is enough to show that there is a singularity of $F(z)$ in $\mathcal{S}(M)$ for such an $M$ as this implies that there are infinitely many singularities in the annulus $|z|\in[L,L^{k^{2d}}]$.

To this end, using \eqref{F1} as above, since $\alpha=Re^{i\vartheta}$ is a singularity of $F(z)$, $R^{k^{-d}}e^{i\vartheta k^{-d}}$ is a singularity of $F(z^{k^d})$ and so for some $j_1\in\{0,\ldots,d-1\}$, $R^{k^{-d+j_1}}e^{i\vartheta k^{-d+j_1}}$ is a singularity of $F(z)$. Continuing in this manner, we have that for some $j_1,\ldots,j_m\in\{0,\ldots,d-1\}$, $R^{k^{-md+j_1+\cdots+j_m}}e^{i\vartheta k^{-md+j_1+\cdots+j_m}}$ is a singularity of $F(z)$ for each $m$. Now let $n$ be such that $nd-(j_1+\cdots+j_n)\in[(M-1)d,Md);$ note that since $md-(j_1+\cdots+j_m)$ grows at least by $1$ and at most by $d$ for each increment in $m$ such an $n$ exists in $[M,Md]$. Set $$\alpha_M:=R^{k^{-nd+j_1+\cdots+j_n}}e^{i\vartheta k^{-nd+j_1+\cdots+j_n}}.$$ Since $L^{k^{Md}} < R<L^{k^{(M+1)d}}$, we have $$L^{k^{Md-nd+j_1+\cdots+j_n}} < R^{k^{-nd+j_1+\cdots+j_n}}<L^{k^{(M+1)d-nd+j_1+\cdots+j_n}},$$ and since $nd-(j_1+\cdots+j_n)\in[(M-1)d,Md)$, we thus have $$L < R^{k^{-nd+j_1+\cdots+j_n}}<L^{k^{(M+1)d-(M-1)d}}=L^{k^{2d}}.$$ Thus $|\alpha_M|\in[L,L^{k^{2d}}]$. Also, $$0<\arg\alpha_M=\frac{\vartheta}{k^{nd-(j_1+\cdots+j_n)}}\leq\frac{2\pi}{k^{(M-1)d}}.$$ Hence $\alpha_M\in\mathcal{S}(M)$. Since $\alpha_M$ is a singularity, we obtain the result.
\end{proof}

\begin{lemma} Let $k\geq 2$ be an integer and $F(z)\in\mathbb{C}[[z]]$ be a $k$-Mahler function. If $F(z)$ is entire, then $F(z)$ is a polynomial.
\end{lemma}

\begin{proof} Let $k\geq 2$ be an integer and $F(z)\in\mathbb{C}[[z]]$ be an entire $k$-Mahler function satisfying $$\sum_{j=0}^{d} a_j(z)F(z^{k^j})=0,$$ for $a_j(z)\in\mathbb{C}[z]$ with $a_0(z)a_d(z)\neq 0$. As in the proof of Lemma \ref{merofinite}, write \begin{equation}\label{F2} F(z^{k^d})=-\sum_{j=0}^{d-1} \frac{a_j(z)}{a_d(z)}F(z^{k^j}).\end{equation} 

Pick $L>1$ such that all of the zeros of $a_d(z)$ are in the open disk $B(0,L)$ of radius $L$ centered at the origin. Notice that since the $a_i(z)$ are polynomials, there is an $N>1$ and a constant $C>1$ such that for $|z|\geq L$, we have \begin{equation}\label{aiadbound} \max_{0\leq i\leq d-1}\left\{\left|\frac{a_i(z)}{a_d(z)}\right|\right\}<C|z|^N;\end{equation} in particular, the value $N=\max_{0\leq i\leq d-1}\{\deg a_i(z), 2\}$ is sufficient.

For $\ell\geq 0$ denote $$M_\ell:=\max\left\{|F(z)|:|z|=L^{k^\ell}\right\},$$ where $L$ is as chosen above. Using \eqref{F2}, \eqref{aiadbound}, and the Maximum Modulus Theorem, we have for $j\geq d$ that $$M_j\leq d \cdot C\big(L^{k^{j-d}}\big)^N M_{j-1}\leq C d L^{Nk^{j}}M_{j-1}.$$ Thus recursively, we have for each $n\geq d$ that $$M_n\leq M_{d-1}(C d)^n L^{Nk^{n+1}}.$$ But since $L>1$, this implies that there is some constant $b>0$ such that for $n\geq d$ we have $$M_n\leq L^{bk^n}.$$

Now let $m\geq b+2$ be a natural number, fix an $\alpha\in\mathbb{C}$ and consider $$F^{(m-1)}(\alpha)=\frac{1}{2\pi i}\int_{\gamma_n}\frac{F(z)}{(z-\alpha)^m}dz,$$ where $\gamma_n$ is the circle of radius $L^{k^n}$ with $n$ large enough so that $\alpha$ is inside the circle of radius $L^{k^n}/2$ centered at the origin. Then for all $z$ on $\gamma_n$ we have that $$\frac{|z|}{2}\leq |z-\alpha|.$$ Thus for $n$ large enough, we have $$|F^{(m-1)}(\alpha)|\leq \frac{1}{2\pi} \cdot 2\pi L^{k^n}\cdot \frac{2^mM_n}{(L^{k^n})^m}= \frac{2^mM_n}{(L^{k^n})^{m-1}}\leq 2^mL^{k^n(b-m+1)}.$$ Recall that $m\geq b+2$ so that the above gives that $$|F^{(m-1)}(\alpha)|\leq \frac{2^m}{L^{k^n}}.$$ Since $n$ can be taken arbitrarily large, we have that $F^{(m-1)}(\alpha)=0$. But $\alpha\in\mathbb{C}$ was arbitrary, and so $F^{(m-1)}(z)$ is identically zero; hence $F(z)$ is a polynomial.
\end{proof}

\begin{theorem}\label{main1} Let $k\geq 2$ be an integer and $F(z)\in\mathbb{C}[[z]]$ be a $k$-Mahler function. If $F(z)$ has only finitely many singularities, then $F(z)$ is a rational function.
\end{theorem}

\begin{proof} Let $k\geq 2$ be an integer and $F(z)\in\mathbb{C}[[z]]$ be a $k$-Mahler function satisfying \begin{equation}\label{F3} \sum_{j=0}^{d} a_j(z)F(z^{k^j})=0,\end{equation} for $a_j(z)\in\mathbb{C}[z]$ with $a_0(z)a_d(z)\neq 0$. If $F(z)$ has only finitely many singularities, then there is a non-zero polynomial $q(z)\in\mathbb{C}[z]$ such that $q(z)F(z)$ is entire. For $j\in\{0,\ldots,d-1\}$ set $$q_j(z):=\frac{1}{q(z^{k^j})}\prod_{i=0}^d q(z^{k^i})\in\mathbb{C}[z].$$ Multiplying \eqref{F3} by $\prod_{i=0}^d q(z^{k^i})\in\mathbb{C}[z]$ we then have that $$\sum_{j=0}^{d} a_j(z)q_j(z)q(z^{k^j})F(z^{k^j})=0,$$ where since $q(z)$ is not identically zero we have that $a_0(z)q_0(z)a_d(z)q_d(z)\neq 0.$ Hence $q(z)F(z)$ is an entire $k$-Mahler function and thus, by the preceding lemma, a polynomial. This proves that $F(z)$ is a rational function.
\end{proof}

\begin{proof}[Proof of Theorem \ref{Df}] A $D$-finite series has finitely many singularities; namely, if $F(z)$ satisfies $p_0(z) F(z) + p_1(z) F'(z) + \cdots + p_m(z) F^{(m)}(z) = 0$ then each singularity of $F(z)$ is a zero of $p_m(z)$.
An application of Theorem \ref{main1} provides the desired result.
\end{proof}

The following corollary is a result of Nishioka \cite[Theorem~5.1.7]{N1996}.

\begin{corollary}[Nishioka \cite{N1996}] Let $k\geq 2$ be an integer and $F(z)\in\mathbb{C}[[z]]$ be a $k$-Mahler function. If $F(z)$ is algebraic, then $F(z)$ is a rational function.
\end{corollary}

\begin{proof} An algebraic series is $D$-finite; see Stanley \cite[Theorem~2.1]{S1980}.
\end{proof}

\section{A P\'olya--Carlson type result}



\begin{proof}[Proof of Theorem \ref{PC}] Suppose that $F(z)\in\mathbb{C}[[z]]$ is $k$-Mahler and not rational. By the structure theorem, we have that $F(z)$ has only polar singularities in the unit disk. Note that we have already shown that if $F(z)$ is meromorphic, then it is a rational function, so we may suppose that there is some non-polar singularity of $F(z)$. By Lemma \ref{merofinite} we have that there are infinitely many non-polar singularities of $F(z)$ on the unit circle. 

Let $S$ be the closure of the non-polar singularities of $F(z)$ on the unit circle. We will show that $S$ is the entire unit circle. Towards a contradiction, suppose there are points $\beta,\gamma\in S$ with $\arg\beta<\arg\gamma$ such that the points on the small arc of the unit circle strictly between $\beta$ and $\gamma$ are not singularities of $F(z)$.

Let \begin{equation}\label{pFd}\sum_{i=0}^d p_i(z)F(z^{k^i})=0\end{equation} be a non-trivial Mahler functional equation for $F(z)$ which is minimal with respect to $d$, and define the vector space $$V:=\sum_{i\geq 0}\mathbb{C}(z)F(z^{k^i}).$$ 

It is quite easy to see that $V$ has dimension $d$ as a $\mathbb{C}(z)$-vector space since $d$ is minimal with respect to the relation \eqref{pFd}. Indeed, suppose that $$\sum_{i=0}^n q_i(z)F(z^{k^i})\in V$$ with $q_n(z)\neq 0$ and $n\geq d$. Then multiplying by $1$ and subtracting zero using \eqref{pFd}, we have that \begin{align*} \sum_{i=0}^n q_i(z)&F(z^{k^i}) = \sum_{i=0}^{n-1}q_i(z)F(z^{k^i})\\
&+\frac{1}{p_d(z^{k^{n-d}})}\left(p_d(z^{k^{n-d}})q_n(z)F(z^{k^n})-q_n(z)\sum_{i=0}^d p_i(z^{k^{n-d}})F(z^{k^{n-d+i}})\right)\\
&\quad\qquad = \sum_{i=0}^{n-1}q_i(z)F(z^{k^i})-\frac{1}{p_d(z^{k^{n-d}})}\left(q_n(z)\sum_{i=0}^{d-1} p_i(z^{k^{n-d}})F(z^{k^{n-d+i}})\right)\\
&\quad\qquad \in\sum_{i=0}^{n-1}\mathbb{C}(z)F(z^{k^{i}}).
\end{align*} Continuing in this manner shows that $$\sum_{i=0}^n q_i(z)F(z^{k^i})\in\sum_{i=0}^{d-1}\mathbb{C}(z)F(z^{k^{i}}).$$ Thus $\dim_{\mathbb{C}(z)}V\leq d.$ Since $d$ was chosen minimally so that \eqref{pFd} holds, we have that $\dim_{\mathbb{C}(z)}V=d$.

Similarly, for all integers $L>0$, the $\mathbb{C}(z)$-vector space $$\sum_{i=0}^{d-1}\mathbb{C}(z)F(z^{k^{i+L}})$$ is $d$-dimensional, and since it is a subspace of $V$ it is equal to $V$. 

Thus since $F(z)\in V$, there are then polynomials $q_{0,L}(z),\ldots,$ $q_{d,L}(z)$ with $q_{0,L}(z)$ nonzero such that \begin{equation}\label{qFL} q_{0,L}(z)F(z)=\sum_{i=1}^d q_{i,L}(z)F(z^{k^{i+L-1}}).\end{equation}

Since $S$ is the closure of the non-polar singularities, and polar singularities are isolated, we can pick a non-polar singularity $\alpha$ of $F(z)$ as close as we wish to $\beta$. Pick such an $\alpha$ and an $L>0$ big enough so that $\alpha\omega^j$ is in the arc between $\beta$ and $\gamma$ for $j=1,\ldots,d+1$ where $\omega:=e^{2\pi i/k^L}.$ Notice that $q_{0,L}(z)F(z)$ has a non-polar singularity at $z=\alpha$ since $q_{0,L}(z)$ is nonzero.

Define \begin{multline*}W:=\Big\{(s_1(z),\ldots,s_d(z))\in\mathbb{C}(z)^d:\\ \sum_{i=1}^d s_i(z)F(z^{k^{i+L-1}})\mbox{ has at most a polar singularity at $z=\alpha$}\Big\}.\end{multline*} Note that $W$ is a $\mathbb{C}(z)$-vector space. Then sending $z\mapsto \omega^jz$ in \eqref{qFL} for $j=1,\ldots,d+1$, we have
\begin{equation}\label{qFeq}
	q_{0,L}(\omega^jz)F(\omega^jz)=\sum_{i=1}^d q_{i,L}(\omega^jz)F((\omega^jz)^{k^{i+L-1}}).
\end{equation}
By construction, the left-hand side of \eqref{qFeq} has no non-polar singularity at $z=\alpha$.
Since $\omega$ is a $k^L$th root of unity, $F((\omega^jz)^{k^{i+L-1}}) = F(z^{k^{i+L-1}})$ for all $j$ and all $i\geq 1$, and thus $$\mathbf{u}_j(z):=\left(q_{1,L}(\omega^jz),\ldots,q_{d,L}(\omega^jz)\right)\in W$$ for $j=1,\ldots,d+1$.
Since $\dim_{\mathbb{C}(z)}W<d+1$, there is some $t\in\{1,\ldots,d+1\}$ such that $$\mathbf{u}_t(z)=g_1(z)\mathbf{u}_1(z)+\cdots+g_{t-1}(z)\mathbf{u}_{t-1}(z).$$ In other words, we have that \begin{equation}\label{star} q_{i,L}(\omega^tz)=g_1(z)q_{i,L}(\omega z)+\cdots+g_{t-1}(z)q_{i,L}(\omega^{t-1}z)\end{equation} for each $i=1,\ldots,d$. 

Replacing $z$ with $\omega z$, using \eqref{star} gives that $$\mathbf{u}_{t+1}(z):=\left(q_{1,L}(\omega^{t+1}z),\ldots,q_{d,L}(\omega^{t+1}z)\right)\in W.$$ Continuing, we get that $\left(q_{1,L}(\omega^{k^L}z),\ldots,q_{d,L}(\omega^{k^L}z)\right)\in W.$ Since $\omega^{k^L}=1$, we have that $$\left(q_{1,L}(z),\ldots,q_{d,L}(z)\right)\in W,$$ and thus \eqref{qFL} gives that $q_{0,L}(z)F(z)$ is non-singular at $z=\alpha$, a contradiction. This proves the theorem.
\end{proof}

\section{Acknowledgements}

We thank Jean-Paul B\'ezivin for pointing out his paper to us and the referee for a careful reading.
The research of J.~P.~Bell was supported by NSERC grant 31-611456 and the research of M.~Coons was supported in part by a Fields--Ontario Fellowship and NSERC.


\end{document}